\documentclass[reqno]{amsart}
\usepackage[utf8]{inputenc}
\usepackage{gensymb}
\usepackage{amsmath}
\usepackage{amssymb}
\usepackage{amsthm}
\usepackage{enumitem}
\usepackage{hyperref}
\usepackage{graphicx}
\usepackage{hyperref}

\usepackage{bbm}

\renewcommand{\phi}{\varphi}

\newcommand{\N}{\mathbb{N}}
\newcommand{\Z}{\mathbb{Z}}

\title{On practical sets and $A$-practical numbers}
\author{Andrzej Kukla and Piotr Miska}
\address{
	Institute of Mathematics\\
	Faculty of Mathematics and Computer Science\\
	Jagiellonian University in Cracow\\
	ul. {\L}ojasiewicza 6, 30-348 Krak\'ow
	Poland\\
 }
\email{andrzej.kukla@student.uj.edu.pl, piotr.miska@uj.edu.pl}

\newcommand{\modd}[1]{\ \normalfont{\text{(mod }}#1)}

\newcommand{\mycomment}[1]{}

\newtheorem{theorem}{Theorem}[section]
\newtheorem{lemma}[theorem]{Lemma}
\newtheorem{hyp}[theorem]{Hypothesis}

\theoremstyle{definition}
\newtheorem{que}[theorem]{Question}
\newtheorem{remark}[theorem]{Remark}

\newtheorem{defi}[theorem]{Definition}
\newtheorem{cor}[theorem]{Corollary}
\newtheorem{ex}[theorem]{Example}
\newtheorem{obs}[theorem]{Observation}

\begin{document}

\thanks{During the preparation of this article Andrzej Kukla was a participant of the tutoring programme under the Excellence Initiative at the Jagiellonian University and was a scholarship holder of the Programme for outstanding students at the Faculty of Mathematics and Computer Science of the Jagiellonian University. Piotr Miska is supported by the Grant of the Polish National Science Centre No. UMO-2019/34/E/ST1/00094.}

\maketitle

\begin{abstract}
    Let $A$ be a set of positive integers. We define a positive integer $n$ as an $A$-practical number if every positive integer from the set $\left\{1,\ldots ,\sum_{d\in A, d\mid n}d\right\}$ can be written as a sum of distinct divisors of $n$ that belong to $A$. Denote the set of $A$-practical numbers as $\text{Pr}(A)$. The aim of the paper is to explore the properties of the sets $\text{Pr}(A)$ (the form of the elements, cardinality) as $A$ varies over the power set of $\N$. We are also interested in the set-theoretic and dynamic properties of the mapping $\mathcal{PR}:\mathcal{P}(\N)\ni A\mapsto\text{Pr}(A)\in\mathcal{P}(\N)$.
\end{abstract}

\section{Introduction}\label{sec1: introduction}

We say that a positive integer $n$ is practical if and only if every integer between $1$ and $\sigma(n)$ can be represented as a sum of some distinct divisors of $n$. The concept of practical numbers was firstly introduced by A. K. Srinivasan in 1948 \cite{Srinivasan} and was later studied by several mathematicians. In particular, B. M. Stewart in 1954 \cite{Stewart} and W. Sierpiński in 1955 \cite{Sierpinski}, independently from each other, proved the Characterization Theorem for practical numbers, now known as the Stewart's criterion.
\begin{theorem}[Stewart's criterion]
    Let $n\ge 2$ be a positive integer and let $n=p_1^{\alpha_1}p_2^{\alpha_2}...p_k^{\alpha_k}$ be its prime factorization, where $p_1<p_2<...<p_k$. Then $n$ is a practical number if and only if 
    $$p_1=2\ \text{ and }\ p_{i+1}\leq \sigma\left(p_1^{\alpha_1}...p_i^{\alpha_i}\right)+1,\ \text{ for }\ i\in\{1,...,k-1\}.$$
\end{theorem}
The above theorem is useful in the study of practical numbers. For example, it allows one to prove that the set of practical numbers is closed under multiplication or that any number has a practical multiple. Also many properties of practical numbers are similar to the ones prime numbers have. In fact, several practical analogues of theorems concerning prime numbers have already been proved, for example analogues of the Prime number theorem \cite{Weingartner}, the Goldbach conjecture \cite{Melfi} or the Twin Prime conjecture \cite{Melfi}. 

There are a few generalizations of the concept of practical numbers, for example N. Schwab and L. Thompson introduced $f$-practical numbers \cite{Thompson}, which are defined through values of a multiplicative function $f$ on the divisors of positive integers. Also S. A. Baddai considered $t$-practical numbers \cite{Baddai}, for which divisors are weighted with integer coefficients smaller than or equal to $t$. In this paper we introduce a different approach to generalizing the notion of practical numbers through restricting the set of divisors that can be used in representations of integers.

This paper is structured as follows. First, we introduce the idea of practical sets, which is a set theoretical analogue of the concept of practical numbers. Throughout all the paper $\N$ stands for the set of positive integers and $\N_0$ denotes the set of non-negative integers.
\begin{defi}
    Let $A\subset \N$. Define $S_A$ to be the sum of all elements from $A$ (we set $S_{\emptyset}=0$). That is
    $$S_A:=\sum_{a\in A}a.$$
    We say that $A\subset\N$ is a \textit{\textbf{practical set}} if and only if every non-negative integer smaller than or equal to $S_A$ can be represented as a sum of distinct elements from $A$.
\end{defi}
Such definition extends the notion of practical numbers onto the subsets of the set of natural numbers. In particular, $n\in\N$ is practical if and only if $D(n)$ is a practical set, where $D(n)$ is the set of all positive divisors of $n$. Similar concepts to this appear in literature, for example any infinite practical set is a special case of a complete sequence \cite{Hoggat}. Also G. Wiseman in notes of a few of his integer sequences on OEIS refers to practical sets as ''complete sets'' \cite{Wiseman}, similarly to Y. O. Hamidoune, A. S. Lladó and O. Serra in \cite{Serra}. Despite this, we decided to hold to the "practical sets" nomenclature in order to emphasize their connection to practical numbers. 

In Section 2 we expand the idea of practical sets by proving several theorems, such as the \nameref{3: practical sets characterization}. In particular many of the facts we prove are similar to the properties of practical numbers themselves. 

In Section 3 we exploit the concept of practical sets and the machinery proved in Section 2 to define and study the following generalization of practical numbers.
\begin{defi}
Let $A\subset\N$. We say that \textit{$n$ is $A$-practical} if and only if
$$D(n)\cap A \text{ is practical.}$$
We also define $\Pr(A)$ as the set of all $A$-practical numbers. For $A=\N$ we obtain the exact definition of practical numbers.
\end{defi}
Further we study basic properties of these numbers and focus on answering the following question.
\begin{que}\label{1: family characterization question}
    Can we characterize families of sets $\mathcal{A}$ such that for any two sets $A_1,A_2\in\mathcal{A}$ the sets of $A_1$-practical numbers and $A_2$-practical numbers are the same?
\end{que}

In Section 4 we take a closer look at the mapping $A\mapsto \Pr(A)$. In particular we investigate its image, continuity and dynamic properties.

\section{Practical sets}\label{sec3: practical sets}

In this section we prove some general properties of practical sets, many of which are very similar to the ones practical numbers have. Let us recall the definition of a practical set. 
\begin{defi}
Let $A$ be a subset of $\N$. We say that $A$ is \textbf{\textit{practical}} if and only if
$$\forall_{0\leq k\leq S_A}\exists_{D\subset A}:k=\sum_{d\in D}d.$$
\end{defi}
We begin with stating a few simple observations.
\begin{obs}
    Let $A\subset\N$.
    \begin{enumerate}
        \item $\emptyset$ is practical.
        \item If $\# A\geq 1$ and $A$ is practical, then $1\in A$.
        \item If $\# A\geq 2$ and $A$ is practical, then $2\in A$.
        \item If $A$ is an infinite practical set and $B\supset A$, then $B$ is practical.
    \end{enumerate}
\end{obs}
It is easy to produce new practical sets on the basis of practical sets. In particular, adding sufficiently small elements to a practical set also yields a practical set.
\begin{lemma}[First Extension lemma]\label{3: extension lemma}
    Let $A\subset \N$ be a practical set and $n\in \N$. Then $$A\cup\{n\}\text{ is practical }\Longleftrightarrow n\leq S_A+1$$
\end{lemma}
\begin{proof}If $n\in A$, then $A\cup\{n\}=A$ is practical. Also, if $A$ is infinite, then $A\subset A\cup\{n\}$, so $A\cup\{n\}$ is practical. Assume that $A$ is finite and $n\notin A$. \begin{enumerate}
    \item[($\Rightarrow$)] For the sake of contradiction assume that $n>S_A+1$. Since $A\cup\{n\}$ is practical and $S_A+1 \leq S_A+n = S_{A\cup\{n\}}$, there exists $D\subset A\cup\{n\}$ such that $S_A+1$ is a sum of elements of $D$. Each element of $D$ must be smaller than or equal to $S_A+1$, which means that $n$ cannot be an element of $D$. Therefore $D$ must be contained in $A$, which is a contradiction since $$S_A+1 = S_D\leq S_A < S_A+1.$$
    \item[($\Leftarrow$)] Let $k\leq S_{A\cup\{n\}}$. $A$ is a subset of $A\cup\{n\}$ and since $A$ is practical we know that $k\leq S_A$ can be represented as sum of distinct elements of $A\cup\{n\}$. Consider $S_A<k\leq S_{A\cup\{n\}}$. Since $S_{A\cup\{n\}}=S_A+n$ and $n\leq S_A+1$, we have
    $$-1\leq S_A-n< k-n\leq S_{A\cup\{n\}}-n=S_A.$$
    Therefore $k-n$ is representable as sum of distinct elements from $A$, hence $k$ is representable as the same sum with added $n$. Hence $A\cup\{n\}$ is practical.
\end{enumerate} 
\end{proof}
In particular we have the following characterization of finite practical sets
\begin{theorem}[Characterization Theorem]\label{3: practical sets characterization}
    Let $A\subset\N$, where $A=\{a_1,...,a_m\}$, $a_1<...<a_m$ and $m\in\N$. $A$ is practical if and only if
    $$a_1=1\quad\text{ and }\quad\forall_{i\in\{1,...,m-1\}}:a_{i+1}\leq \sum_{j=1}^{i}a_j+1.$$
\end{theorem}
\begin{proof}
    \begin{enumerate}
        \item[($\Rightarrow$)] Assume that $A$ is practical. Since $A\neq \emptyset$ we know that $1\in A$ and since $a_1$ is the smallest element of $A$, $a_1$ must be equal to 1. For the sake of contradiction assume that 
        $$\exists_{i\in\{1,...,m-1\}}:a_{i+1}>\sum_{j=1}^{i}a_j+1.$$
        Then $S_A\geq a_{i+1}>\sum_{j=1}^{i}a_j+1$ and $\sum_{j=1}^{i}a_j+1$ cannot be represented as sum of distinct elements from $A$, which contradicts the fact that $A$ is practical.
        \item[($\Leftarrow$)] Let $A_r=\{a_1,...,a_r\}$ for $r=\{1,...,m\}$. We proceed with mathematical induction with respect to $r$ to prove that $A_r$ is practical. We know that $A_1=\{1\}$ is practical. Now assume that $A_{r-1}$ is practical. We assumed that $a_r\leq \sum_{j=1}^{r-1}a_j+1=  S_{A_{r-1}}+1$, so by \nameref{3: extension lemma} we know that $A_r$ is practical. Therefore $A_m$ is practical.
    \end{enumerate}
\end{proof}
\begin{remark}
    The theorem above is also true for infinite sets $A\subset\N$.
\end{remark}
\begin{cor}\label{3: remove max from practical set}
    If $A\subset\N$ is finite and practical, then $A\setminus\{\max A\}$ is also practical.
\end{cor}
Let us present some examples of families of practical sets.
\begin{ex}
    \begin{enumerate}
        \item The set $\{1,2,...,m\}$ is practical for any $m\in\N$. In particular $\N$ is practical.
        \item The set $\{F_2,F_3,...,F_m\}$ is practical for any $m\geq 2$, where $F_n$ is the $n$-th Fibonacci number defined with the following recurrence relation: $F_0=0, F_1=1, F_{n+2}=F_{n+1}+F_n$. In particular the set of all Fibonacci numbers is practical.
        \item The set that consists of $n$ first practical numbers is practical. Also the set that consists of all practical numbers is practical.
    \end{enumerate} 
\end{ex}
The \nameref{3: practical sets characterization} also allows us to greatly simplify the proof of one of the properties of practical numbers, namely that practical numbers are also quasi-practical.
\begin{defi}
    Let $n\in\N$ and define \textbf{\textit{aliquot sum}} of $n$ as the sum of all proper divisors of $n$. That is
    $$s(n):=\sum_{d\mid n,d\neq n}d=\sum_{d\in S(n)}d,$$
    where $S(n)$ is the set of all proper divisors of $n$. We say that $n$ is a \textbf{\textit{quasi-practical number}} if and only if every natural number smaller than or equal to $s(n)$ can be represented as a sum of distinct proper divisors of $n$. The notion of quasi-practical numbers was introduced by Y. Peng and K. P. S. Bhaskara Rao in \cite{Rao}.
\end{defi} 
\begin{theorem}\label{3: quasi-practical = prime + practical}
    A positive integer $n$ is quasi-practical if and only if $n$ is practical or $n$ is prime. 
\end{theorem}
    
\begin{proof}
\begin{enumerate}
    \item[($\Rightarrow$)] Assume that $n$ is quasi-practical and so $S(n)$ is practical. If $n$ is prime, then the implication holds. Assume that $n$ is not a prime number. Then $n$ has at least two positive proper divisors and since $S(n)$ is practical, $2$ must be an element of $S(n)$. This means that $n$ is even, therefore $\frac n2$ is also an element of $S(n)$. Practicality of $S(n)$ implies that there exists $D\subset S(n)$ such that $\sum_{d\in D}d=\frac n2-1$. Each element of $D$ is smaller than $\frac n2$, so we have
    $$n -1= \frac n2 + \frac n2 -1 = \frac n2 + \sum_{d\in D}d \leq S_{S(n)}.$$
    It follows from the \nameref{3: extension lemma} that $S(n)\cup\{n\}=D(n)$ is practical, which means that $n$ is practical.
    \item[($\Leftarrow$)] If $n$ is prime, then $s(n)=1$, so $n$ is quasi-practical. Now assume that $n$ is practical. Therefore $D(n)$ is also practical. Since $D(n)=S(n)\cup\{n\}$ and $n=\max D(n)$, it follows from Corollary \ref{3: remove max from practical set} that $S(n)$ is also practical, which means that $n$ is quasi-practical. 
\end{enumerate}
\end{proof}
Practical sets possess a few properties that are very similar to the ones practical numbers have.
\begin{theorem}
    Let $A\subset\N$ be finite. The set $A$ is practical if and only if
    $$\forall_{a\in A}\exists_{D\subset A}:a-1=\sum_{d\in D}d.$$
\end{theorem}
\begin{proof}
    \begin{enumerate}
        \item[($\Rightarrow$)] $A$ is practical, so each number smaller than $S_A$ is represented as some sum of distinct elements from $A$. In particular, for each $a\in A$ the number $a-1$ is represented as such sum.
        \item[($\Leftarrow$)] Let $A=\{a_1,...,a_m\}$, where $a_1<...<a_m$ and $A_r=\{a_1,...,a_r\}$ for $r\in\{1,...,m\}$. We proceed by induction with respect to $r$ to prove that $A_r$ is practical. For the base case we prove that $A_1=\{a_1\}$ is practical. We know that there exists $D\subset A$ such that $\sum_{d\in D}d=a_1-1$, so each $d\in D$ is smaller than $a_1$, but $a_1$ is the smallest element of $A$, hence $D$ must be empty and $a_1-1=0$. Hence $A_1=\{1\}$, therefore it is practical. Now assume that $A_{r-1}$ is practical. We know that there exists $D\subset A$ such that $\sum_{d\in D}d=a_r-1$, so each $d\in D$ is smaller than $a_r$. Therefore $D\subset A_{r-1}$ and
        $$a_r = \sum_{d\in D}d+1\leq \sum_{a\in A_{r-1}}a+1 = \sum_{j=1}^{r-1}a_j+1,$$
        so from the \nameref{3: extension lemma} it follows that $A_{r-1}\cup\{a_r\}=A_r$ is practical.
    \end{enumerate}
\end{proof}
\begin{cor}\label{3: equivalence of set practicality}
    Let $\emptyset\neq A\subset \N$ be finite. The set $A$ is practical if and only if
    $$\forall_{0\leq k<\max A}\exists_{D\subset A}:k=\sum_{d\in D}d.$$
\end{cor}
\begin{cor}
    If $n\in\N$, then all numbers smaller than $n$ are representable as sums of distinct divisors of $n$ if and only if all numbers smaller than $\sigma(n)$ are representable as sums of distinct divisors of $n$.
\end{cor}
The following lemma will be useful for proving two more properties of practical sets similar to ones practical numbers also possess. For two sets $A,B\subset\N$ we define their product as 
$$AB:=\{ab:a\in A,b\in B\}.$$
\begin{lemma}[Second Extension lemma]\label{3: second extension lemma}
    Let $A\subset\N$ be practical and $B=\{1,b_1,...,b_m\}$, where $b_1<...<b_m$. Also let $B_0=\{1\}$ and $B_r=\{1,b_1,...,b_r\}$ for $r\in\{1,...,m\}$. $AB$ is practical if and only if 
    $$\forall_{r\in\{1,...,m\}}:b_r\leq S_{AB_{r-1}}+1$$
\end{lemma}
\begin{proof}
    \begin{enumerate}
        \item[($\Rightarrow$)] For the sake of contradiction assume that there exists $r_0\in\{1,...,m\}$ such that $b_{r_0}>S_{AB_{r_{0}-1}}+1$. We can notice that 
        $$AB = AB_{r_0-1}\cup (AB\setminus AB_{r_0-1})=AB_{r_0-1}\cup A\{b_{r_0},b_{r_0+1},...,b_m\},$$
        and $\min(A\{b_{r_0},b_{r_0+1},...,b_m\})=b_{r_0}$, so from the assumption $b_{r_0}>S_{AB_{r_{0}-1}}+1$. It follows that $S_{AB_{r_{0}-1}}+1$ cannot be represented as a sum of distinct elements of $AB$, which contradicts the assumption of practicality of $AB$.
        \item[($\Leftarrow$)] Let $\max A = m_A$. We proceed by induction with respect to $r$ to prove that $AB_{r}$ is practical for $r\in\{0,...,m\}$. For the base case we notice that $AB_0=A$, which indeed is practical. Now we assume that $AB_{r-1}$ is practical and consider $0\leq k < m_Ab_r$. Then
        $$k=qb_r+t,\ \text{ where }\ q\in \{0,1,...,m_A-1\}\ \text{ and }\ t\in\{0,1,...,b_r-1\}.$$
        If $q=0$, then $k=t<b_r$, therefore $k$ is representable as a sum of distinct elements of $AB_r$. Now assume that $q>0$. Both $A$ and $AB_{r-1}$ are practical and we know that $t\leq b_r-1\leq S_{AB_{r-1}}$, thus there exist sets $D_q\subset A$, $D_t\subset AB_{r-1}$ such that 
        $$q=\sum_{d_q\in D_q}d_q\ \text{ and }\ t=\sum_{d_t\in D_t}d_t.$$
        If $q=0$, then Since $t\leq b_r-1$ we know that $\forall_{d_t\in D_t}:d_t<b_r$, so each element $d_qb_r$ is greater than each $d_t$. Hence, $k$ is representable as a sum of distinct elements of $AB_r$ as $$k=\sum_{d_q\in D_q}d_qb_m + \sum_{d_t\in D_t}d_t.$$
        It follows from the Corollary \ref{3: equivalence of set practicality} that $AB_r$ is practical.
    \end{enumerate}
\end{proof}
\begin{remark}
    The above lemma is also true for infinite sets $A,B \subset \N$.
\end{remark}
Here is a special case of the above lemma.
\begin{cor}\label{3: product of sets theorem - special case with A cup B}
    Let $A\subset\N$ be practical and $\{1\}\subset B\subset\N$ be finite. Then $$A\cup B \text{ is practical}\  \Longrightarrow \ AB \text{ is practical}.$$
\end{cor}
\begin{proof}
Let $B=\{1,b_1,...,b_m\}$, where $b_1<...<b_m$. Since $A\cup B$ is practical it follows from the \nameref{3: practical sets characterization} that
$$\forall_{r\in\{1,...,m\}}:b_r\leq S_A+\sum_{j=1}^{r-1}b_j+1\leq S_{AB_{r-1}}+1.$$
Therefore $A$ and $B$ satisfy assumptions of the \nameref{3: second extension lemma}, hence $AB$ is practical.
\end{proof}
With the help of the Corollary \ref{3: product of sets theorem - special case with A cup B} we can prove that, similarly to practical numbers, product of a practical set and its subset is practical, but also that product of two practical sets is a practical set.
\begin{cor}
    If $A\subset\N$ is practical and $\{1\}\subset B\subset A$, then $AB$ is also practical.
\end{cor}
\begin{proof}
    Since $B\subset A$ we have $A\cup B = A$, which is practical, so from the Corollary \ref{3: product of sets theorem - special case with A cup B} it follows that $AB$ is practical.
\end{proof}
\begin{cor}
    If $A,B\subset\N$ are practical, then $AB$ is also practical.
\end{cor}
\begin{proof}
    If $A$ or $B$ is infinite, then $AB$ is practical, because $A\subset AB$ or $B\subset AB$, respectively. Assume that both $A$ and $B$ are finite. Without loss of generality we can assume that $\max B\leq \max A$. Thus each $b\in B$ satisfies $b\leq \max A\leq S_A+1$, so it follows from the \nameref{3: extension lemma} that $A\cup B$ is practical. Corollary \ref{3: product of sets theorem - special case with A cup B} implies that $AB$ is practical.
\end{proof}

\section{$A$-practical numbers}\label{sec4: A-practical numbers}
In the following section we prove basic properties of $A$-practical numbers and partly answer Question 1.4 stated in Section 1.  Let us recall the definition of $A$-practical numbers. 

\begin{defi}
Let $A\subset\N$. We say that \textit{$n$ is $A$-practical} if and only if
$$D(n)\cap A \text{ is practical.}$$
We also define $\Pr(A)$ as the set of all $A$-practical numbers. For $A=\N$ we obtain the exact definition of practical numbers.
\end{defi}

We first note that the absence of $1$ or $2$ in the set $A\subset\N$ allows us to indicate exactly which natural numbers are $A$-practical.
\begin{obs}\label{4: 1 and 2 characterization} Let $A\subset\N$.
\begin{enumerate}
\item 
If $1\not\in A$, then $n\in\N$ is $A$-practical if and only if $D(n)\cap A=\emptyset$. 
\item 
If $1\in A$ and $2\not\in A$, then $n\in\N$ is $A$-practical if and only if $D(n)\cap A=\{1\}$. 
\end{enumerate}
\end{obs}
\begin{proof}
\begin{enumerate}
    \item Let $n\in \N$. If $D(n)\cap A=\emptyset$, then $D(n)\cap A$ is practical, so $n$ is $A$-practical. If $D(n)\cap A\neq\emptyset$, then $D(n)\cap A$ cannot be practical, since $1\notin D(n)\cap A$. Therefore $n$ is not $A$-practical.
    \item Let $n\in \N$. If $D(n)\cap A=\{1\}$, $D(n)\cap A$ is practical, so $n$ is $A$-practical. If $D(n)\cap A\neq \{1\}$, then $D(n)\cap A$ cannot be practical, since $2\notin D(n)\cap A$. Therefore $n$ is not $A$-practical.
\end{enumerate}
\end{proof}
\begin{theorem}\label{4: D(n) theorem}
Let $n,m\in\N$. The following conditions are equivalent:
\begin{enumerate}[label=(\roman*)]
    \item $n$ is $D(m)$-practical,
    \item $m$ is $D(n)$-practical,
    \item $\gcd(n,m)$ is practical.
\end{enumerate}
\end{theorem}
\begin{proof} Let us prove the ($(i)\Longleftrightarrow (iii)$) part, $(ii)\Longleftrightarrow (iii)$ is analogous. Notice that 
$$D(n)\cap D(m) = D(\gcd(n,m))$$
so $n$ is $D(m)$-practical $\Longleftrightarrow$ $D(n)\cap D(m)$ is practical $\Longleftrightarrow$ $D(\gcd(n,m))$ is practical $\Longleftrightarrow$ $\gcd(n,m)$ is practical.
\end{proof}

\begin{cor}
Based on Theorem \ref{4: D(n) theorem} we can notice that all natural numbers are $D(2^n)$-practical for any $n\in\N$. Therefore, if $D(2^\infty):=\{2^n:n=0,1,2,...\}$, then all natural numbers are $D(2^\infty)$-practical. Also, $2^m$ is $D(n)$-practical for any $n,m\in\N$;

\end{cor}

Most of subsets $A\subset\N$ do not have the same property as $D(2^n)$, mainly due to prime numbers.
\begin{obs}\label{4: prime observation}
    Let $A\subset \N$ and $p\in\mathbb{P}_{\geq 3}$. If $p\in A$, then $p$ is not $A$-practical and if $p\notin A$, then $p$ is $A$-practical.
\end{obs}
\begin{proof}
    If $p\in A$, then $D(p)\cap A$ is equal to $\{p\}$ or $\{1,p\}$. Since $p\geq 3$, neither of these sets is practical, hence $p$ is not $A$-practical. If $p\notin A$, then $D(p)\cap A$ is equal to $\{1\}$ or $\emptyset$. Since both of these sets are practical, $p$ is $A$-practical.
\end{proof}
Above observation can be generalized with the help of quasi-practical numbers.
\begin{obs}\label{4: S(n) subset of A}
    Let $A\subset\N$ and $n\in\mathbb{N}$.
    \begin{enumerate}
        \item If $D(n)\subset A$, then $n$ is $A$-practical if and only if $n$ is practical.
        \item If $S(n)\subset A$ and $n\notin A$, then $n$ is $A$-practical if and only if $n$ is quasi-practical.
    \end{enumerate}
\end{obs}
\begin{proof}
Follows from definitions of practical numbers and quasi-practical numbers and from the Theorem \ref{3: quasi-practical = prime + practical}.   
\end{proof}

\subsection{Expansion and removal theorems}
Let us recall the Question \ref{1: family characterization question} from Section 1
\begin{que}
    Can we characterize families of sets $\mathcal{A}$ such that
    $$\forall_{A_1,A_2\in \mathcal{A}}:\Pr(A_1)=\Pr(A_2).$$
\end{que}
In the following subsection we partly answer this question by introducing expansion and removal theorems.

\begin{defi}
    We define a partial order $\prec$ on the power set of $\N$ as follows
    $$A\prec A'\Longleftrightarrow A\subseteq A'\text{ and }\Pr(A)\subseteq \Pr(A').$$
\end{defi}
\begin{que}\label{4: minimal maximal question}
    If $A\subset\N$, do there always exist sets $A'$ and $A''$ such that $A'\prec A\prec A''$, $A'$ is minimal and $A''$ is maximal with respect to the order $\prec$?
\end{que}
\begin{obs}\label{4: minimal maximal corrolary} Let $A\subset\N$.
\begin{enumerate}
        \item We have $A\prec (A\cup\{1\})$.    
        \item $A$ is minimal if and only if $\forall_{a\in A}:D(a)\cap A = \{a\}$ and $A\neq\{1\}$. In particular, if $A\subset\mathbb{P}$, then $A$ is minimal.
        \item If $A$ is minimal and $$A\subset A'\subset (\{a_1^{\beta_1}a_2^{\beta_2}\ldots a_n^{\beta_n}: n\in\N, a_1,a_2,\ldots, a_n\in A,\beta_1,\beta_2,\ldots,\beta_n\in \N_0\}\setminus\{1\}),$$
        then $\Pr (A)=\Pr (A')$. In particular, $A\prec A'$.
        
\end{enumerate}
\end{obs}
    \begin{proof}
\begin{enumerate}
        \item If $1\notin A$, then based on the Observation \ref{4: 1 and 2 characterization} we know that
        $$\Pr(A)=\{m\in\N:D(m)\cap A=\emptyset\}.$$
        Therefore if $n\in\Pr(A)$, then $n\in\Pr(A\cup\{1\})$, because $D(n)\cap(A\cup\{1\})=\{1\}$.
     \item \begin{enumerate}
         \item[($\Rightarrow$)] Assume that $\forall_{a\in A}:D(a)\cap A = \{a\}$ and $A\neq\{1\}$. We have that $1\notin A$, so from the Observation \ref{4: 1 and 2 characterization} it follows that $$\Pr(A)=\{m\in\N:D(m)\cap A=\emptyset\}.$$
        Now let $B\subsetneq A$. Therefore $\exists_{a\in A}:a\notin B$. Notice that $a$ is $B$-practical, because there are no divisors of $a$ in $B$, but $a$ is not $A$-practical, because $D(a)\cap A=\{a\}$ and $\{a\}$ is not practical. Hence $\Pr(B)\nsubseteq \Pr(A)$, thus $B\nprec A$, therefore $A$ is minimal.
        \item[($\Leftarrow$)] Assume that $A$ is minimal and for the sake of contradiction suppose that $A=\{1\}$ or $\exists_{a\in A}:D(a)\cap A\supsetneq \{a\}$. If the former is true, then $\emptyset\prec A$, which is a contradiction with minimality of $A$. If the latter is true, then 
        $$\exists_{b\neq a}:b\in D(a)\cap A.$$
        If $1\in A$, then $(A\setminus\{1\})\prec A$, contradiction. If $1\notin A$, then based on the Observation \ref{4: 1 and 2 characterization} we know that
        $$\Pr(A)=\{m\in\N:D(m)\cap A=\emptyset\}.$$
        If $m$ is divisible by $a$, then $m$ is divisible by $b$, so $A\setminus\{a\}\prec A$, contradiction.
     \end{enumerate}
        \item Notice that $1\notin A'$, so from Observation \ref{4: 1 and 2 characterization} $$\Pr(A')=\{m\in\N:D(m)\cap A'=\emptyset\}.$$
        If $D(m)\cap A'=\emptyset$ for some $m\in\mathbb{N}$, then $D(m)\cap A=\emptyset$ since $A\subset A'$. On the other hand, if $D(m)\cap A=\emptyset$ for some $m\in\mathbb{N}$, then $D(m)\cap A'=\emptyset$ since
        $$A'\subset (\{a_1^{\beta_1}a_2^{\beta_2}\ldots a_n^{\beta_n}: n\in\N, a_1,a_2,\ldots, a_n\in A,\beta_1,\beta_2,\ldots,\beta_n\in (\N\cup\{0\})\}\setminus\{1\}).$$
        Therefore $\Pr(A)= \Pr(A')$, so $A\prec A'$.
\end{enumerate}        
    \end{proof}
As it turns out, the answer to the Question \ref{4: minimal maximal question} is positive.
\begin{theorem}\label{4: minimal maximal theorem}
    If $A\subset\mathbb{N}$, then there exist sets $A'\subset A\subset A''$ such that 
    $$A'\prec A\prec A'',$$
    where $A'$ is minimal and $A''$ is maximal with respect to the order $\prec$.
\end{theorem}
    \begin{proof}
        We set $A'=\{a\in A:a>1,D(a)\cap A\subset\{1,a\}\}\subset A$. Set $A'$ satisfies
        $$\forall_{a' \in A'}:D(a')\cap A'=\{a'\},$$
        so it follows from \ref{4: minimal maximal corrolary} that $A'$ is minimal and that $A'\prec A$. Now consider $\mathcal{A}=\{B\subset\N:A\prec B\}$ and let $\mathcal{C}\subset \mathcal{A}$ be a chain with respect to $\prec$. Now we have
        \begin{equation*}        \Pr\left(\bigcup_{C\in\mathcal{C}}C\right)=\bigcup_{C\in\mathcal{C}}\Pr(C).   
        \end{equation*}
        To see that, notice that for any $n\in\N$ there exists $C_n\in\mathcal{C}$ such that 
        $$\forall_{C\succ C_n}:D(n)\cap C = D(n)\cap C_n.$$
        In particular $D(n)\cap \left(\bigcup_{C\in\mathcal{C}}C\right) = D(n)\cap C_n$, so
        $$n\in\Pr\left(\bigcup_{C\in\mathcal{C}}C\right) \Longleftrightarrow n\in\Pr(C_n).$$
        Also $\Pr(C_n)\subset\bigcup_{C\in\mathcal{C}}\Pr(C)$ so:
        \begin{enumerate}[label=\textbullet]
            \item if $n\in\Pr(C_n)$, then $n\in \bigcup_{C\in\mathcal{C}}\Pr(C)$,
            \item if $n\notin\Pr(C_n)$, then $\forall_{C\prec C_n}:n\notin\Pr(C)$, but since $\forall_{C\succ C_n}:D(n)\cap C = D(n)\cap C_n$ we have $\forall_{C\succ C_c}:n\notin\Pr(C)$, so $n\notin \bigcup_{C\in\mathcal{C}}\Pr(C)$.
        \end{enumerate}
        Therefore $$n\in\bigcup_{C\in\mathcal{C}}\Pr(C) \Longleftrightarrow n\in\Pr(C_n).$$
        We know that $\bigcup_{C\in\mathcal{C}}C\in\mathcal{A}$, so from (1) it follows that $\bigcup_{C\in\mathcal{C}}C$ is an upper bound of $\mathcal{C}$. Zorn's lemma implies that there exists some maximal element $A''$ in the set $\mathcal{A}$.
    \end{proof}
\mycomment{One idea was that we might have an implication
$$n \text{ is }A\text{-practical}\Rightarrow n \text{ is }D(\text{lcm}A)\text{-practical},$$
but this is generally false due to observation \ref{4: prime observation}, because possibly many new primes are generated in $D(\text{lcm}A)$. It seems like the set $A$ needs to contain all the proper divisors of $n$ to make the implication
$$
m \text{ is }A\text{-practical}\Rightarrow m \text{ is }(A\cup\{n\})\text{-practical}
$$
true, but we have managed to prove it with an additional assumption that $\sigma(n)\geq 2n-1$.}

\begin{theorem}\label{4: expand theorem}
    Let $A\subset\N$ and $n\in\N\setminus\mathbb{P}_{\geq 3}$ such that $S(n)\subset A$, $n\notin A$ and $\sigma(n)\geq 2n-d_0-1$ where $d_0:=\inf(A\setminus D(n))$. Then $$A\prec (A\cup\{n\}).$$
\end{theorem}
\begin{proof}
    Let $m$ be an $A$-practical number. If $n\nmid m$, then $m$ is trivially $(A\cup\{n\})$-practical. Consider $m=an$ for $a\in\N$. Since $an$ is $A$-practical we know that $D(an)\cap A\text{ is practical}$ and we want to check whether $D(an)\cap (A\cup\{n\})\text{ is practical}.$
    Note that $D(an)\cap (A\cup\{n\}) = (D(an)\cap A)\cup\{n\}$. 
    \begin{enumerate}[label=\arabic*\degree]
        \item Assume that $\sigma(n)\geq 2n-1$. Since $S(n)\subset A$, we also have $S(n)\subset (D(an)\cap A)$, so 
        $$S_{D(an)\cap A}\geq n-1\ \Rightarrow\ S_{D(an)\cap A}+1\geq n.$$
        It follows from the \nameref{3: extension lemma} that $D(an)\cap (A\cup\{n\})$ is practical.
        \item Assume that $2n-d_0-1\leq \sigma(n)<2n-1$. It follows that $n$ is not practical and since $n$ is not a prime, by Observation \ref{4: S(n) subset of A} $S(n) = D(n)\cap A$ is not practical. Since $D(an)\cap A$ is practical there exists some $d\in\N$ such that $d\in((D(an)\setminus D(n))\cap A)$. We have $d\geq d_0$, and $S(n)\subset (D(an)\cap A)$, so
        $$S_{D(an)\cap A}\geq d+n-d_0-1\geq n-1\Rightarrow\ S_{D(an)\cap A}+1\geq n.$$
        It follows from the \nameref{3: extension lemma} that $D(an)\cap (A\cup\{n\})$ is practical.
        
    \end{enumerate}
\end{proof} 
Following corollary shows that any finite subset $A\subset \N$ can be extended to an infinite set.
\begin{cor} Let $A\subset \N$. If $2^n$ is the smallest power of $2$ that is not an element of $A$, then $A\prec (A\cup\{2^n\})$.
\end{cor}
\begin{proof} Let $2^n$ be the smallest power of $2$ that is not an element of $A$. $\sigma(2^n)=2^{n+1}-1$, so $2^n$ satisfies assumption of the Theorem \ref{4: expand theorem}. Hence $A\prec(A\cup\{2^n\})$.
\end{proof}

\begin{ex}\label{4: 1 2 3 example}
Let $A=\{1,2,3\}$ and $$A'=\{1,2,3,4,6,8,12,16,24,...\}=\{2^n3^\varepsilon:n\in\N\cup\{0\},\varepsilon\in\{0,1\}\}.$$ We have $\Pr(A)\subset\Pr(A')$, because $\sigma(2^n3^\varepsilon)\geq 2^{n+1}3^{\varepsilon}-1$ for any $n\in\N$ and $\varepsilon\in\{0,1\}$.
\end{ex}
We can notice that proof of the Theorem \ref{4: expand theorem} depends on the fact that $an$ is $(A\cup\{n\})$-practical for $an$ being $A$-practical. Therefore the following theorem is also true.
\begin{theorem}\label{4: expand theorem 2}
 Let $A\subset\N$ and $n\in\N$ such that $n\notin A$. Then 
 $$\left[\forall_{a\in\mathbb{N}}:(an\in\Pr(A))\Rightarrow(an\in\Pr(A\cup\{n\})\right]\Longleftrightarrow A\prec (A\cup\{n\}).$$ 
\end{theorem}
\begin{proof}

\begin{enumerate}
    \item[$(\Rightarrow)$] Let $m$ be an $A$-practical number. If $n\nmid m$, then $m$ is trivially $(A\cup\{n\})$-practical. Consider $m=an$ for $a\in\N$. Since $an$ is practical, we know that $an$ is $(A\cup\{n\})$-practical, therefore $\Pr(A)\subset \Pr(A\cup\{n\})$ and $A\prec (A\cup\{n\})$.
    \item[$(\Leftarrow)$] Consider $an$ for some $a\in\N$ and assume that $an$ is $A$-practical. Since $A\prec (A\cup\{n\})$ we have $an\in\Pr(A)\subset\Pr(A\cup\{n\})$.
\end{enumerate}
    
\end{proof}
\begin{remark}
In Theorem \ref{4: expand theorem 2}, if $bn\in \Pr(A)$ implies $bn\in\Pr(A\cup\{n\})$, then $abn\in \Pr(A)$ implies $abn\in\Pr(A\cup\{n\})$ for any $a\in\N$. It follows from the \nameref{3: extension lemma}. In particular if $n\in\Pr(A)$ and $S_{D(n)\cap A}+1\geq n$, then $A\prec (A\cup\{n\})$.
\end{remark}
\begin{remark}
In Theorem \ref{4: expand theorem 2} it is sufficient to check the assumption $\forall_{a\in\mathbb{N}}:(an\in\Pr(A))\Rightarrow(an\in\Pr(A\cup\{n\}))$ only for $a\in\mathbb{N}$ that are least common multiples of elements from $A\setminus D(n)$ (including the least common multiple of the empty set, being equal to $1$).
\end{remark}

Theorem \ref{4: expand theorem 2} also implies that when we want to extend set $A$ to set $B$ we can do it "step by step" by adding the smallest missing element.

\begin{theorem}\label{4: step by step theorem}
    Let $A\subsetneq B\subset\N$ and $k:=\min(B\setminus A)$. Then
    $$(A\prec B)\Rightarrow (A\prec (A\cup\{k\})).$$
\end{theorem}
\begin{proof}
    For the sake of contradiction assume that $(A\nprec (A\cup\{k\}))$. From Theorem \ref{4: expand theorem 2} it follows that $$\exists_{a\in\N}:ak\in\Pr(A)\text{ and }ak\notin\Pr(A\cup\{n\}).$$
    Now notice that since $A\prec B$ we have $ak\in\Pr(B)$ and since $k\in B$ we have $k\in D(ak)\cap B$, so $S_{D(ak)\cap B}>k-1$. It follows from the \nameref{3: extension lemma} that $k>S_{D(ak)\cap A}+1$. Every divisor of $ak$ from $B\setminus A$ is greater than or equal to $k$ so $k-1$ cannot be represented as sum of divisors of $ak$ from $B$, which contradicts the fact that $D(ak)\cap B$ is practical.
\end{proof}

\begin{ex}
    In example \ref{4: 1 2 3 example} we showed that $$A=\{1,2,3\}\prec \{2^n3^\varepsilon:n\in \N\cup\{0\},\varepsilon\in\{0,1\}\}=A'.$$ It follows from Theorem \ref{4: expand theorem 2} that we can extend $A'$ to 
    $$A''=\{2^n3^m:n,m\in\N\cup\{0\}\}.$$
    First we extend $A'$ with $9$. We can do that, because $9$ is not $A'$-practical, $18$ is $A'$-practical and 18 is $A''$-practical. Then we can extend by numbers of the form $2^n\cdot 9$, $n\in\N$, because all of them are practical and the extended set contains all of their divisors. Analogously we extend $A'$ by next powers of 3 followed by powers of 2 multiplied by those powers of $3$.

    From Theorems \ref{4: expand theorem 2} and \ref{4: step by step theorem} it also follows that $A''$ is maximal. 
    \begin{proof}
        Suppose that $\exists_{B\subset\N}:A''\prec B$ and set $k:=\min(B\setminus A'')$. We have $k=2^r3^lm$, where $(m,6)=1$ and $m\geq 5$. Now consider $$k_0=\left\{\begin{array}{ll}
            k & \text{if }r\geq 1 \\
            2k & \text{if }r=0
        \end{array}\right..$$
        Since $D(k_0)\cap A=D(2^{r'}3^l)$, where $r'=\max\{ 1,r\}$, and $2^{r'}3^l$ is $A''$-practical we know that $k_0$ is $A''$-practical. We want to prove that $S_{A''}(k_0)<k-1$. From Theorem \ref{4: expand theorem 2} it would follow that $k_0$ is not $A''\cup\{k\}$-practical and we would arrive at contradiction with $A''\prec B$.
        \begin{enumerate}[label=\arabic*\degree]
            \item If $k_0=k$, then
            $$S_{A''}(k_0)=\sigma(2^r)\sigma(3^l)=\frac{(2^{r+1}-1)(3^{l+1}-1)}{2}<2^r3^{l+1}<5\cdot 2^r3^l-1\leq k-1.$$
            \item If $k_0=2k$, then $r=0$, so
            $$S_{A''}(k_0)= \sigma(2)\sigma(3^{l})=\frac{3^{l+2}-3}{2}=\frac{9}{2}\cdot3^l-\frac32 < 5\cdot3^l -1 = k-1.$$
        \end{enumerate}
    \end{proof}
\end{ex}

\begin{ex}\label{4: 1 2 3 example extended}
    We can manually check that 
    $$\Pr(\{1,2,3\})=\Pr(\{2^n3^m:n,m\in \N_0\})=\{k\in\N:k\not\equiv 3\modd{6}\}.$$
    We can order the second set as follows
    $$\{2^n3^m:n,m\in \N_0\}=\{1,2,3,4,6,8,9,12,16,18,24,27,...\}=\{a_1,a_2,a_3,...\},$$
    where $a_1<a_2<a_3<...$. 
    Let us denote $A_k=\{a_1,a_2,...,a_k\}$. Theorem \ref{4: step by step theorem} implies that 
    $$\forall_{k\in\N, k\ge 3}:\Pr(A_k)=\{k\in\N:k\not\equiv 3\modd{6}\}.$$
\end{ex}
\begin{que}
    Can we explain the behaviour of the above sets in terms of expansion and removal theorems?
\end{que}
From Theorems \ref{4: expand theorem} and \ref{4: expand theorem 2} we know that
$$\forall_{k\in\N}:A_k\prec A_{k+1}.$$
The following theorem explains why $\Pr (A_{k+1})\subset\Pr(A_k)$.

\begin{theorem}
Let $A\subset \N$ be finite and $n\in\N$ such that $n\notin A$ and $n>\max A$. Then $$\Pr(A\cup\{n\})\subset\Pr(A).$$   
\end{theorem}
\begin{proof}
Let $m$ be an $(A\cup\{n\})$-practical number. If $n\neq m$, then $m$ is trivially $A$-practical. Consider $m=an$ for $a\in\N$. Then $D(an)\cap(A\cup\{n\})$ is practical. Since $n>\max A$, we have $n=\max(D(an)\cap(A\cup\{n\}))$ and $D(an)\cap(A\cup\{n\})=(D(an)\cap A)\cup\{n\}$, so by the Observation \ref{3: remove max from practical set} it follows that $D(an)\cap A$ is practical, therefore $an$ is $A$-practical.
\end{proof}
It is also possible to remove elements smaller than the maximum of a set.
\begin{theorem}
    Let $A\subset\N$ be finite and $n\in\N$ be a practical number such that $S(n)\subset A$ and $n\notin A$. If $\sup(A\setminus S(n))\leq s(n)+1$, then 
    $$\forall_{a\in\N}:an\in\Pr(A).$$
    In particular $\Pr(A\cup\{n\})=\Pr(A)$.
\end{theorem}
\begin{proof}
    Let $a\in\N$. We note that $S(n)\subset D(an)\cap A$ and since $n$ is practical we know that $S(n)$ is practical. Let
    $$(D(an)\cap A)\setminus S(n)=\{a_1,a_2,...,a_l\},$$
    where $a_i\leq s(n)+1$ for any $i=1,...,l$. From the mathematical induction with respect to $r$ and from the \nameref{3: extension lemma} it follows that $S(n)\cup\{a_1,...,a_r\}$ is practical. Therefore, for $r=l$, we get that $D(an)\cap A$ is practical, hence $an$ is $A$-practical. In particular $\Pr(A\cup\{n\})\subset\Pr(A)$ and from theorem \ref{4: expand theorem} it follows that $\Pr(A)\subset\Pr(A\cup\{n\})$, thus $\Pr(A)=\Pr(A\cup\{n\})$.
\end{proof}
\begin{que}
    In previous theorems we considered only finite sets $A\subset\N$. Is it possible to establish some results concerning the reduction of an infinite set $A$ to some finite set $A'$ with $\Pr(A)\subset \Pr(A')$?
\end{que}

\section{The mapping $\mathcal{PR}$}

\subsection{The image of $\mathcal{PR}$}
In this section we analyse different properties of the following mapping.
\begin{defi}
    Let $$\mathcal{PR}:\mathcal{P}(\N)\ni A\mapsto \Pr(A)\in\mathcal{P}(\N).$$
\end{defi}
We can notice that $\mathcal{PR}$ is neither surjective nor injective. For the former we know that $1$ is $A$-practical for any $A\subset\N$ and for the latter we have already seen examples of pairs of sets $A\neq B$ such that $\Pr(A)=\Pr(B)$ (see Example \ref{4: 1 2 3 example extended}).

It turns out that sets of type $A=D(2^n)$ and $A=\emptyset$ are the only ones for which every natural number is $A$-practical.
\begin{obs}\label{4: D(2^n)}
    Let $A\subset\N$. If all natural numbers are $A$-practical, then $A=\emptyset$ or there exists $n\in\N\cup\{0,\infty\}$ such that $A=D(2^n)$.
\end{obs}
\begin{proof}
    For the sake of contradiction, suppose that $A$ is not of the form $D(2^n)$ and it is not empty. Set $k:=\min (A\setminus D(2^\infty))$ and let $r=\nu_2(k)$. Then
    $$D(k)\cap A=A'\cup\{k\},\text{ where }A'\subset\{1,2,...,2^r\}.$$
    Since $k\geq 3\cdot2^r>2^r+1\geq S_{A'}$, it follows from the \nameref{3: practical sets characterization} that $D(k)\cap A$ is not practical, therefore $k$ is not $A-$practical.
\end{proof}
\begin{que}
    If $A\subset\N$, what is the cardinality of $\Pr(A)$?
\end{que}
 On one hand if $A\subset\N$ is finite, then there are infinitely many $A$-practical numbers, in particular numbers that are coprime with each element of $A$ are $A$-practical. On the other hand, there exist sets $A\subset\N$ such that there are only finitely many $A$-practical numbers.

\begin{theorem}
    For all $k\in\N$ there exists $A_k\subset\N$ such that $\{1,...,k-1\}\subset A_k$, $k\notin A_k$ and both $\N\setminus A_k$ and $\Pr(A_k)$ are finite.
\end{theorem}
\begin{proof}
    We set $A_k$ such that $\N\setminus A_k=\left\{k,...,\frac{k(k-1)}{2}+1\right\}$. If $m>\frac{k(k-1)}{2}+1$, then $m\in A_k$, hence 
    $$D(m)\cap A_k=A_1\cup A_2,$$
    where $A_1\subset \{1,2,...,k-1\}$ and $S_{A_1}\leq \frac{k(k-1)}{2}$, but also $A_2\neq \emptyset$ and $S_{A_2}>\frac{k(k-1)}{2}+1$.
    Therefore $D(m)\cap A_k$ cannot be practical, so $m$ cannot be $A_k$-practical; hence $\Pr(A_k)$ is finite.
\end{proof}
We can use the above construction to find sets $A\subset \N$ such that $\Pr(A)$ is finite but $\N\setminus A$ is infinite.
\begin{theorem}
    For all $k\in\N$ there exists $A_k\subset\N$ such that $\{1,...,k-1\}\subset A_k$, $k\notin A_k$, $\N\setminus A_k$ is infinite and $\Pr(A_k)$ is finite.
\end{theorem}
\begin{proof}
    We set $A_k$ such that $\N\setminus A_k=\left\{k,...,\frac{k(k-1)}{2}+1\right\}\cup\left\{\frac{nk(k-1)}{2}+2n:n\geq 2\right\}$. If $m>\frac{k(k-1)}{2}+1$, then $m\in A_k$ or $\frac{k(k-1)}{2}+2\mid m$, hence 
    $$D(m)\cap A_k=A_1\cup A_2,$$
    where $A_1\subset \{1,2,...,k-1\}$ and $S_{A_1}\leq \frac{k(k-1)}{2}$, but also $A_2\neq \emptyset$ and $S_{A_2}>\frac{k(k-1)}{2}+1$.
    Therefore $D(m)\cap A_k$ cannot be practical, so $m$ cannot be $A_k$-practical; hence $\Pr(A_k)$ is finite.
\end{proof}
\begin{que}
    Can we find a necessary and sufficient condition for $B\subset \N$ to be an element of $\text{Im}(\mathcal{PR})$?
\end{que}
\begin{obs}
    If $B\in\text{Im}(\mathcal{PR})$ and $p$ is a prime number such that $p^m\in B$ for some $m\in\N$, then $p^k\in B$ for all $0\leq k<m$.
\end{obs} 
\begin{theorem}
    Let $f\in \Z[x]$ and $B\subset\N$ be of the form $B=f(\Z)\cap\N$ or $B=f(\N)\cap\N$. Assume that $B$ is infinite. Then $B\in\text{Im}(\mathcal{PR})$ if and only if $B=\N$ or $B=\{2n-1: n\in\N\}$.
\end{theorem}

\begin{proof}
    Assume that there exists $A\subset\N$ such that $B=\text{Pr}(A)$.
    
    Start with the case of $\deg f\geq 2$. Since $B\in\text{Im}(\mathcal{PR})$, we have $1\in f(\Z)$. Let $p_0$ be the least odd prime number such that $p_0\not\in B$ but $ap_0\in B$ for some $a\in\N$. Such a number exists as there are $O\left(x^{\frac{1}{\deg f}}\right)$ positive integers at most equal to $x$ represented as values of $f$ at integer arguments and there are $\Omega\left(\frac{x}{\log x}\right)$ prime numbers $p\leq x$ such that $p\mid f(m)$ for some $m\in\N$ (by Frobenius density theorem). Since $p_0\not\in B$, there must be $p_0\in A$. Let $n\in B$ be an odd positive integer (such an $n$ exists because of existence of $m_1,m_2\in\N$ such that $2\nmid f(m_1)$ and $p_0\mid f(m_2)$ and Chinese remainder theorem). Then $p_0\in D(n)\cap A$ and $2\not\in D(n)\cap A$, so $S_{D(n)\cap A}\geq p_0>2$ but $2$ cannot be written as sum of some elements of from $D(n)\cap A$. This means that $n\not\in\text{Pr}(A)=B$, which is a contradiction.

    We are left with the case of $\deg f=1$. Then $B=\{an+b: n\in\N_0\}$ for some $a,b\in\N$. Because $1\in B$, we have $b=1$. Let $a\geq 3$. By Dirichlet theorem on prime numbers in arithmetic progressions there exists an odd prime number $p_0\equiv -1\pmod{a}$. Then $p_0^2\equiv 1\pmod{a}$ so $p_0^2\in B$. On the other hand, as $p_0\not\in B$ and $p_0\neq 2$, there must be $p_0\in A$. Thus $S_{D(p_0^2)\cap A}\geq p_0>2$ and $2$ cannot be written as sum of some elements of from $D(n)\cap A$. This means that $n\not\in\text{Pr}(A)=B$, which is a contradiction.

    In order to finish the proof of the theorem, we note that $\N=\text{Pr}(\varnothing)$ and $\{2n-1: n\in\N\}=\text{Pr}(\{2\})$.
\end{proof}
Similarly to \ref{4: D(2^n)} one can ask whether there is any strict subset $A\subsetneq\N$ such that $\Pr(A)$ is the set of all practical numbers. Our hypothesis is that there is no such strict subset.
\begin{hyp}
    If $A\subset\N$ such that $\Pr(A)=\Pr(\N)$, then $A=\N$.
\end{hyp}

\subsection{Continuity of $\mathcal{PR}$}

Remember that the membership of a positive integer $n$ to the set $\text{Pr}(A)$ is equivalent to the practicality of the set $A\cap D(n)\subset A\cap\{1,\ldots ,n\}$. Thus, we have the following fact.

\begin{obs}\label{4: independence on greater elements}
    Let $A,B\subset\N$ be such that $A\cap\{1,\ldots ,n\}=B\cap\{1,\ldots ,n\}$. Then $k\in\text{Pr}(A)\Longleftrightarrow k\in\text{Pr}(B)$ for each $k\in\{1,\ldots ,n\}$. In other words $$\text{Pr}(A)\cap\{1,\ldots ,n\}=\text{Pr}(B)\cap\{1,\ldots ,n\}.$$
\end{obs}

The power set $\mathcal{P}(\N)$ can be identified with Cartesian product $\{0,1\}^\N$ of countably many copies of a discrete set $\{0,1\}$. Hence, we may equip $\mathcal{P}(\N)$ in the product topology. This topology is metrizable. One of possible metrics inducing this topology is the one given by the formula
$$d(A,B)=(\inf(A\div B))^{-1},\, A,B\in\mathcal{P}(\N),$$
where $A\div B$ denotes the symmetric difference of the sets $A$ and $B$ and we set the conventions that $\inf\varnothing =+\infty$ and $\frac{1}{+\infty}=0$. Since $A\cap\{1,\ldots ,n\}=B\cap\{1,\ldots ,n\}$ if and only if $d(A,B)<\frac{1}{n}$, we can reformulate Observation \ref{4: independence on greater elements} in the language of metric spaces.

\begin{cor}
    Let $A,B\subset\N$. Then $$d(\text{Pr}(A),\text{Pr}(B))\leq d(A,B).$$ In other words, the mapping $\mathcal{PR}$ is Lipschitzian with Lipschitz constant $1$ with respect to the metric $d$.
\end{cor}

Actually, $1$ is the optimal Lipschitz constant. Indeed, let $N\geq 2$ be a positive integer and $A=\varnothing$, $B=\{N\}$. Then $\text{Pr}(A)=\N$ and $\text{Pr}(B)=\{qN+r: q\in\N_0, r\in\{1,\ldots ,N-1\}\}$. Consequently, $d(\text{Pr}(A),\text{Pr}(B))=d(A,B)=\frac{1}{N}$.

Moreover, there is no metric on $\mathcal{P}(\N)$ such that $\mathcal{PR}$ is Lipschitzian with Lipschitz constant less than $1$ with respect to this metric. Existence of such a metric would imply existence of a (unique) fixed point of the mapping $\mathcal{PR}$ by Banach fixed point theorem as $\mathcal{P}(\N)$, being a compact topological space, is complete with respect to any metric inducing the topology on $\mathcal{P}(\N)$. However, $\mathcal{PR}$ has not a fixed point as $p\in A\div\text{Pr}(A)$ for each odd prime number $p$ and subset $A$ of $\N$. 

Let $(A_j)_{j\in\N}$ be a convergent sequence of subsets of $\N$ with respect to the product topology on $\mathcal{P}(\N)$. This means that for each $n\in\N$ there exists a $k=k(n)\in\N$ such that $n\in A_j$ for all $j\geq k$ or $n\not\in A_j$ for all $j\geq k$. As a consequence,
$$\N=\left(\bigcup_{k=1}^\infty\bigcap_{j=k}^\infty A_j\right)\cup\left(\bigcup_{k=1}^\infty\bigcap_{j=k}^\infty\N\backslash A_j\right)=(\liminf_{j\to\infty}A_j)\cup (\liminf_{j\to\infty}\N\backslash A_j),$$
from which we easily conclude that $\liminf_{j\to\infty}A_j=\limsup_{j\to\infty}A_j$. Thus, the convergence of the sequence $(A_j)_{j\in\N}$ coincides with the set-theoretic definition of the limit of this sequence. Since the mapping $\mathcal{PR}$ is continuous, we know that if $(A_j)_{j\in\N}$ is a convergent sequence of subsets of $\N$, then
$$\text{Pr}(\lim_{j\to\infty}A_j)=\lim_{j\to\infty}\text{Pr}(A_j).$$
In particular,
$$\text{Pr}\left(\bigcup_{j=1}^\infty A_j\right)=\lim_{j\to\infty}\text{Pr}(A_j)$$
for an increasing sequence $(A_j)_{j\in\N}$ (with respect to inclusion) and
$$\text{Pr}\left(\bigcap_{j=1}^\infty A_j\right)=\lim_{j\to\infty}\text{Pr}(A_j)$$
for a decreasing sequence $(A_j)_{j\in\N}$. This agrees with the fact contained in the proof of Theorem \ref{4: minimal maximal theorem} that 
        \begin{equation*}        \Pr\left(\bigcup_{C\in\mathcal{C}}C\right)=\bigcup_{C\in\mathcal{C}}\Pr(C) 
        \end{equation*}
for a chain $\mathcal{C}$ of subsets of $\N$ with respect to the partial order $\prec$. Recall that
$$A\prec A'\Longleftrightarrow A\subseteq A'\text{ and }\Pr(A)\subseteq \Pr(A').$$

\subsection{Periodic points of $\mathcal{PR}$}
\begin{que}
    Does there exist a subset of the set of natural numbers that is a periodic point of the mapping $\mathcal{PR}$? In other words does there exist a set $A\subset\N$ with the property that $A=\Pr^k(A)$ for some $k\in\N$?
\end{que}
It turns out that the answer to the above is positive. Moreover the only possible period of such periodic point is 2. We state and prove it in the following section.

We start with stating a simple fact that will be crucial for proving a following result.
\begin{obs}\label{4: S(n) characterization of A-practicality}
    Let $A\subset\N$ and $n\in\N$. 
    \begin{enumerate}
        \item If $S(n)\cap A$ is practical and $S_{S(n)\cap A}\geq n-1$, then $n\in\Pr(A)$.
        \item If $S(n)\cap A$ is practical and $S_{S(n)\cap A}< n-1$, then $n\in A\div \Pr(A)$.
        \item If $S(n)\cap A$ is not practical, then $n\notin \Pr(A)$.
    \end{enumerate}
\end{obs}
\begin{theorem}
    Let $D\in\N$ and $A_1,A_2,\ldots,A_D\subset\N$ be such that $A_{1}=\text{Pr}(A_D)$ and $A_{j+1}=\text{Pr}(A_j)$ for each $j\in\{1,\ldots ,D-1\}$. Then $D=2d$ for some $d\in\N$ and $A_{2j-1}=A_1$, $A_{2j}=A_2$ for each $j\in\{1,\ldots ,d\}$.
\end{theorem}
\begin{proof}
Let $A_{D+1}:=A_1$. Firstly we prove that $D$ must be even. For the sake of contradiction assume that $D=2d+1$. By the Observation \ref{4: prime observation} we know that if $p$ is an odd prime, then $p\in A\div \Pr(A)$ for any $A\subset\N$. Therefore, by induction on $j$ we show that
\begin{itemize}
    \item $p\in A_1\Leftrightarrow p\in A_j$ for odd $j$,
    \item $p\in A_1\div A_j$ for even $j$.
\end{itemize}
In particular, $p\in A_1\div A_{2d+2} = A_1\div A_1 = \varnothing$, which is a contradiction. Now let $D=2d$ for some $d\in\N$. We will prove that $A_{2j-1}=A_1$ and $A_{2j}=A_2$ for each $j\in \{1,...,d\}$ by showing that 
\begin{equation}
B_{2j-1}^{(k)}=B_1^{(k)}\ \text{ and }\ B_{2j}^{(k)}=B_2^{(k)}    
\end{equation}
for each $j\in \{1,...,d+1\}$ and $k\in\N$, where $B_{j}^{(k)} := A_{j}\cap\{1,...,k\}$. We proceed with mathematical induction with respect to $k$. We know that $1\in\Pr(A)$ for any $A\subset\N$, so $1\in A_j$ for any $j\in\{2,...,D\}$ and $1\in A_1$ because $A_1=\Pr(A_D)$. Therefore $B_{j}^{(1)}=\{1\}$ for any $j\in\{1,...,D\}$. Now let us assume that (2) is true for $k=n-1$. Notice that we only need to prove that
$$B_{2j-1}^{(n)}\cap\{n\}=B_1^{(n)}\cap\{n\}\ \text{ and }\ B_{2j}^{(n)}\cap\{n\}=B_2^{(n)}\cap\{n\}$$
for each $j\in \{1,...,d+1\}$ because induction hypothesis implies that those sets agree on the set $\{1,...,n-1\}$. We make two observations for $j\in\{1,...,D\}$: 
\begin{enumerate}[label=(\roman*)]
    \item $S(n)\cap B^{(n)}_j = S(n)\cap B^{(n-1)}_j$, which follows from the fact that $n\notin S(n)$,
    \item $\Pr\left(B_j^{(n)}\right)\cap\{1,...,n\}=B_{j+1}^{(n)}$, which follows from the Observation \ref{4: independence on greater elements} applied to sets $B_j^{(n)}$ and $A_j$. 
\end{enumerate} We divide the proof into two cases.
\begin{enumerate}[label=\arabic*$\degree$]
    \item Assume that there exists $i\in\{1,...,D\}$ such that $$\left(S(n)\cap B_i^{(n)} \text{ is practical and }S_{S(n)\cap B_i^{(n)}}\geq n-1\right)\text{ or }\left(S(n)\cap B_i^{(n)}\text{ is not practical.}\right)$$ 
    Due to (i) we can change $B_i^{(n)}$ to $B_i^{(n-1)}$ in above expressions. Therefore, from the Observation \ref{4: S(n) characterization of A-practicality} and from the induction hypothesis it follows that $$\forall_{\underset{j\equiv i\modd{2}}{j\in\{1,...,D\}}}:n\in\Pr\left(B_j^{(n)}\right)\ \text{ or }\ \forall_{\underset{j\equiv i\modd{2}}{j\in\{1,...,D\}}}:n\notin\Pr\left(B_j^{(n)}\right),$$
    so (ii) implies that $$\forall_{\underset{j\equiv i+1\modd{2}}{j\in\{1,...,D\}}}:B_{j}^{(n)}=B_{i+1}^{(n)}$$ and since $B_{j+1}^{(n)} = \Pr\left(B_j^{(n)}\right)\cap\{1,...,n\}$ we also have $$\forall_{\underset{j\equiv i\modd{2}}{j\in\{1,...,D\}}}:B_{j}^{(n)}=B_{i}^{(n)}.$$
    \item Assume that for all $i\in\{1,...,D\}:$ $S(n)\cap B_{i}^{(n)}$is practical and $S_{S(n)\cap B_{i}^{(n)}}< n-1$.
    From the Observation \ref{4: S(n) characterization of A-practicality} and from (ii) we know that 
    $$\forall_{i\in\{1,...,D\}}:n\in B_{i}^{(n)}\div B_{i+1}^{(n)}.$$ Since $B_{D+1}^{(n)}=B_{1}^{(n)}$, $n$ is either only in odd indexed sets $B_i^{(n)}$ or only in even indexed sets $B_i^{(n)}$. Therefore the induction hypothesis implies that
    $$B_{2i-1}^{(n)}=B_1^{(n)}\ \text{ and }\ B_{2i}^{(n)}=B_2^{(n)}$$
    for each $i\in\{1,...,D\}$.
\end{enumerate}
\end{proof}

\begin{theorem}\label{4: existance of 2-periodic point}
    There exists a 2-periodic point of the mapping $\mathcal{PR}.$ In other words, there exists a set $A\subset\N$ such that $A=\Pr^2(A)$.
\end{theorem}
\begin{proof}
    Let $I\subset(\N\setminus\{1\})$. We define sets $A$ and $B$ as  unions of  nondecreasing sequences of sets $(A_n)_{n\in\N}$, $(B_n)_{n\in\N}$ where $A_1=B_1=\{1\}$ and
    \begin{enumerate}[label=\arabic*$\degree$]
        \item if $S(n)\cap A_{n-1}$ is practical, $S_{S(n)\cap A_{n-1}}\geq n-1$ and
        \begin{enumerate}[label=1.\arabic*$\degree$]
            \item if $S(n)\cap B_{n-1}$ is practical, $S_{S(n)\cap B_{n-1}}\geq n-1$, then $$A_n=A_{n-1}\cup\{n\}\text{ and }B_n=B_{n-1}\cup\{n\},$$
            \item if $S(n)\cap B_{n-1}$ is practical but $S_{S(n)\cap B_{n-1}}<n-1$, then $$A_n=A_{n-1}\text{ and }B_{n}=B_{n-1}\cup\{n\},$$
            \item if $S(n)\cap B_{n-1}$ is not practical, then $$A_n=A_{n-1}\text{ and }B_{n}=B_{n-1}\cup\{n\},$$
        \end{enumerate}
        \item if $S(n)\cap A_{n-1}$ is practical, $S_{S(n)\cap A_{n-1}}< n-1$ and
        \begin{enumerate}[label=2.\arabic*$\degree$]
            \item if $S(n)\cap B_{n-1}$ is practical, $S_{S(n)\cap B_{n-1}}\geq n-1$, then $$A_n=A_{n-1}\cup\{n\}\text{ and }B_n=B_{n-1},$$
            \item if $S(n)\cap B_{n-1}$ is practical but $S_{S(n)\cap B_{n-1}}<n-1$, then $$\left\{\begin{array}{lll}
               n\in I &\Longrightarrow & A_n=A_{n-1}\text{ and }B_{n}=B_{n-1}\cup\{n\},  \\
               n\notin I &\Longrightarrow& A_n=A_{n-1}\cup\{n\}\text{ and }B_{n}=B_{n-1}, 
            \end{array}\right.$$
            \item if $S(n)\cap B_{n-1}$ is not practical, then $$A_n=A_{n-1}\text{ and }B_{n}=B_{n-1}\cup\{n\},$$
        \end{enumerate}
        \item if $S(n)\cap A_{n-1}$ is not practical and
        \begin{enumerate}[label=3.\arabic*$\degree$]
            \item if $S(n)\cap B_{n-1}$ is practical, $S_{S(n)\cap B_{n-1}}\geq n-1$, then $$A_n=A_{n-1}\cup\{n\}\text{ and }B_n=B_{n-1},$$
            \item if $S(n)\cap B_{n-1}$ is practical but $S_{S(n)\cap B_{n-1}}<n-1$, then
            $$A_n=A_{n-1}\cup\{n\}\text{ and }B_{n}=B_{n-1},$$
            \item if $S(n)\cap B_{n-1}$ is not practical, then $$A_n=A_{n-1}\text{ and }B_{n}=B_{n-1},$$
        \end{enumerate} 
    \end{enumerate}
    for $n\geq 2$. What it means is that we construct $A_n$ and $B_n$ based on the statement of the Observation \ref{4: S(n) characterization of A-practicality} so that $$B_n\cap\{n\}=\Pr(A_{n-1})\cap\{n\}\text{ and }A_n\cap\{n\}=\Pr(B_{n-1})\cap\{n\}.$$
Since $S(n)\cap A_{n-1}=S(n)\cap A_{n}$ and $S(n)\cap B_{n-1}=S(n)\cap B_{n}$ we also have 
\begin{equation}
B_n\cap\{n\}=\Pr(A_{n})\cap\{n\}\text{ and }A_n\cap\{n\}=\Pr(B_{n})\cap\{n\}    
\end{equation}
and from the mathematical induction it follows that 
\begin{equation}
B_n=B_n\cap\{1,...,n\}=\Pr(A_{n})\cap\{1,...,n\}\text{ and }A_n=A_n\cap\{1,...,n\}=\Pr(B_{n})\cap\{1,...,n\}.    
\end{equation}

Indeed, the base case is true due to the fact that $A_1=B_1=\{1\}$ and $\Pr(A_1)=\Pr(B_1)=\N$. For the induction step notice that 
$$A_k\cap \{1,...,k-1\}=A_{k-1}\text{ and }B_k\{1,...,k-1\}= B_{k-1},$$
so from the Observation \ref{4: independence on greater elements} it follows that
\begin{align*}
   &\Pr(A_k)\cap \{1,...,k-1\}=\Pr(A_{k-1})\cap\{1,...,k-1\}\text{ and}\\
   &\Pr(B_k)\cap \{1,...,k-1\}=\Pr(B_{k-1})\cap\{1,...,k-1\}. 
\end{align*}
Thus, if we assume that that (4) is true for $n=k-1$, then it immediately follows from the above and from (3) that (4) is true for $n=k$.

\mycomment{Indeed base case is true due to the fact that $A_1=B_1=\{1\}$ and $\Pr(A_1)=\Pr(B_1)=\N$. Now assume that (3) is true for $n=k-1$. Then 
\begin{equation}
B_{k}\cap\{k\}=\Pr(A_{k})\cap\{k\}\text{ and }A_{k}\cap\{k\}=\Pr(B_{k})\cap\{k\},    
\end{equation}
which, as mentioned earlier, follows from the way $A_k$ and $B_k$ are defined and from the Observation \ref{4: S(n) characterization of A-practicality} and from the fact that $S(k)\cap A_{k-1}=S(k)\cap A_{k}$ and $S(k)\cap B_{k-1}=S(k)\cap B_{k}$. Since $B_k\cap\{1,...,k-1\}=B_{k-1}$ and $A_k\cap\{1,...,k-1\}=A_{k-1}$ it follows from the Observation \ref{4: independence on greater elements} that 
$$\Pr(B_k)\cap\{1,...,k-1\}=\Pr(B_{k-1})\text{ and }\Pr(A_k)\cap\{1,...,k-1\}=\Pr(A_{k-1}),$$
thus induction hypothesis implies that 
$$B_{k}\cap\{1,...,k-1\}=\Pr(A_{k})\cap\{1,...,k-1\}\text{ and }A_{k}\cap\{1,...,k-1\}=\Pr(B_{k})\cap\{1,...,k-1\}.$$
Thesis follows from (4) and the above line.}

Now for $A=\bigcup_{n\in\N}A_n$ and $B=\bigcup_{n\in\N}B_n$ we have 
$$k\in A\Longleftrightarrow k\in A_k\underset{(4)}{\Longleftrightarrow}k\in \Pr(B_k) \Longleftrightarrow k\in\bigcup_{n\in\N}\Pr(B_n)\underset{\ref{4: minimal maximal theorem}}\Longleftrightarrow k\in \Pr(B)$$
and analogously $k\in B\Longleftrightarrow k\in \Pr(A)$, so
$$A=\Pr(B) \text{ and }B= \Pr(A),$$
thus $A=\Pr(B)=\Pr\left(\Pr(A)\right)$.
\end{proof}
\begin{theorem}
    There are continuum $2$-periodic points of the mapping $\mathcal{PR}$.
\end{theorem}
\begin{proof}
    Take two distinct sets $I_1,I_2\subset\mathbb{P}$ and let $p\in I_1\div I_2$ be odd. Let $\left(A_n^{(1)}\right)_{n\in\N}$ and $\left(A_n^{(2)}\right)_{n\in\N}$ be sequences of sets constructed respectively for sets $I_1$ and $I_2$ during the process described in the proof of the Theorem \ref{4: existance of 2-periodic point}. We can notice that any set $\{1\}\subset C\subset\{1,...,p-1\}$ satisfies
    $$S(p)\cap C =\{1\}\text{ is practical but }S_{S(p)\cap C}<p-1,$$
    so when constructing sets $A_p^{(1)}$ and $A_p^{(2)}$ we follow the 2.2$\degree$ case. Since $p\in I_1\div I_2$ we know that $p\in A_p^{(1)}\div A_p^{(2)}$, hence $p\in A^{(1)}\div A^{(2)}$, where $A^{(1)}=\bigcup_{n\in\N}A_n^{(1)}$ and $A^{(2)}=\bigcup_{n\in\N}A_n^{(2)}$.
    
    This proves that if $I_1,I_2\subset\mathbb{P}$ are different, then two sets $A^{(1)}$, $A^{(2)}$ constructed in the process described in the proof of the Theorem \ref{4: existance of 2-periodic point} based on the sets $I_1$, $I_2$, are also different. Since $|\mathcal{P}(\mathbb{P})|=\mathfrak {c}$ and each of the subsets of the set of prime numbers yields a different 2-periodic point of $\mathcal{PR}$, there are continuum of such points.
\end{proof}
\begin{ex}\label{pairexample}
    If we set $I = \N\setminus\{1\}$, then the process described in the proof of Theorem \ref{4: existance of 2-periodic point} yields the following set
    $$A=\{1,2,4,6,8,16,20,24,28,30,32,40,42,56,60,64,72,...\},$$
    where 
    \begin{align*}
        \Pr(A)&=\{1,2,3,4,5,7,8,9,10,11,12,13,14,15,16,17,19,21,22,23,25,...\}\\
        &=\N\setminus\{6, 18, 20, 24, 28, 30, 40, 42, 54, 56, 60, 66, 72, ...\}.
    \end{align*}
    Note that $A\cup\text{Pr}(A)\neq\N$ as $18, 54, 66\not\in A\cup\text{Pr}(A)$.
\end{ex}

\section*{Acknowledgements}

The authors wish to thank Maciej Ulas for discussions that inspired us to explore the subject of $A$-practical numbers deeper.

\end{document}